\documentclass[a4paper,11pt]{amsart}

\pdfoutput=1

\usepackage[text={400pt,660pt},centering]{geometry}

\usepackage{amsthm, amssymb, amsmath, amsfonts, mathrsfs}
\usepackage{mathtools}
\usepackage{scalerel} 
\usepackage{stackrel}
\usepackage[colorlinks=true, pdfstartview=FitV, linkcolor=blue, citecolor=blue, urlcolor=blue,pagebackref=false]{hyperref}
\usepackage{esint} % contains strokedint
\usepackage{MnSymbol} % contains notin inversed (which is \nowns)
\usepackage{booktabs} % nice tables

\usepackage{microtype}

\newtheorem{proposition}{Proposition}
\newtheorem{theorem}[proposition]{Theorem}
\newtheorem{lemma}[proposition]{Lemma}
\newtheorem{corollary}[proposition]{Corollary}

\theoremstyle{remark}
\newtheorem{remark}[proposition]{Remark}

\theoremstyle{definition}

\newtheorem{hypothesis}[proposition]{Hypothesis}

\numberwithin{equation}{section}
\numberwithin{proposition}{section}
\numberwithin{table}{section}

\renewcommand{\ge}{\geqslant}
\renewcommand{\leq}{\leqslant}
\renewcommand{\geq}{\geqslant}

\newcommand{\mcl}{\mathcal}

\newcommand{\A}{\mcl A}

\newcommand{\Ll}{\left}
\newcommand{\Rr}{\right}

\newcommand{\rhs}{right-hand side}
\newcommand{\1}{\mathbf{1}}
\newcommand{\R}{\mathbb{R}}

\newcommand{\N}{\mathbb{N}}

\newcommand{\Zd}{{\mathbb{Z}^d}}

\renewcommand{\epsilon}{\varepsilon}
\renewcommand{\d}{{\mathrm{d}}}

\newcommand{\Ind}[1]{\mathbf{1}_{\left\{#1\right\}}}

\newcommand{\T}{\mathbb{T}}

\newcommand{\x}{\mathbf{x}}
\DeclareMathOperator{\inte}{int}

\title[Antiferromagnetic Potts Model on the Tree]{Power Law Decay at Criticality for the q-State Antiferromagnetic Potts Model on Regular Trees}
\author[C.\ Gu, W.\ Wu, and K.\ Yang]{Chenlin Gu, Wei Wu, Kuan Yang}

\address[Chenlin Gu]{Mathematics Department, NYU Shanghai \& NYU-ECNU Institute of Mathematical Sciences,  China}

\address[Wei Wu]{Mathematics Department, NYU Shanghai \& NYU-ECNU Institute of Mathematical Sciences,  China}

\address[Kuan Yang]{John Hopcroft Center for Computer Science, Shanghai Jiao Tong University, China}

\begin{document}

\begin{abstract}
We present a proof of the power law decay of magnetic moment for the $q$-state antiferromagnetic Potts model on the regular tree at critical temperature, and also justify that the exact exponent is $\frac{1}{2}$. Our proof relies on the assumption of the uniqueness at critical temperature, which has been established for $q=3,4$, and for $q \ge 5$ with large degree.  An iterative contraction inequality is developed for independent interests.

\bigskip

\noindent \textsc{MSC 2010:} 05C99, 60K35, 82B20. 

\medskip

\noindent \textsc{Keywords:} antiferromagnetic Potts model, Gibbs measure, infinite regular tree, critical exponent. 

\end{abstract}
\maketitle

%\tableofcontents

%
%
%
%
%%%%%%%%%%%%%%%%%%%%%%%%%%%%%%%%%%%%%%%%%%%%%%%%%%%%%%%%%%%%%%
%%%%%%%%%%%%%%%%%%%%%%%%%%%%%%%%%%%%%%%%%%%%%%%%%%%%%%%%%%%%%%
%
%
%
%
\section{Introduction}

We study the Gibbs measure of the antiferromagnetic (AF for short) Potts models on regular trees. Let $\T^{d}_{n}$ be the $d$-ary tree, which is a rooted tree of height $n$ and  every vertex except leaves has $d$ children. Here the height means the graph distance from the root to leaves. We then use $\partial \T^{d}_{n}$ for the set of leaves and $\inte(\T^{d}_{n}) := \T^{d}_{n} \setminus \partial \T^{d}_{n}$ its interior. Let $q \inte \N^{+}$ be the number of state for the spin in our model and let $[q] = \{1,2,3, \cdots, q\}$ be all possible states.
%The Gibbs measure plays a central role in statistical physics, and also has prominent applications in combinatorics and computer science.
%Let $G = (V, E)$ be a finite graph. The Potts model has two parameter $q$ and $\beta$, where $q \in \N^{+}$ is the number of states for the spin in the Potts model and $\beta$ is an interaction parameter. Let $[q] = \{1,2,3, \cdots, q\}$. Given a configuration $\sigma \in [q]^V$, define a Hamiltonian
%$H_G(\sigma) = - \sum_{u, v \in V, (u, v)\in E} \Ind{\sigma_u = \sigma_v}$.
%The partition function and the Gibbs measure of the model is defined by $Z_{G,\beta,q} = \sum_{\sigma \in [q]^V} \exp(-\beta H_G(\sigma))$ and $\mu_{G,\beta,q}(\sigma) = \exp(-\beta H_G(\sigma))/Z_{G,\beta,q}$. When $\beta > 0$, it is the ferromagnetic Potts model where the adjacent spins tend to be equal, and when $\beta < 0$, it is the antiferromagnetic Potts model where adjacent spins tend to be different.
%
%Let $\T^{d}_{n}$ be the $d$-ary tree, which is a rooted tree of height $n$ and  every vertex except leaves has $d$ children. Here the height means the graph distance from the root to leaves, and we use $\partial \T^{d}_{n}$ for the set of leaves. 
We define $\sigma \in [q]^{\T^{d}_{n}}$ the spin on $d$-ary tree, and associate it with a Hamiltonian $H^\xi_{\T^{d}_{n}}$
\begin{align}
    H^\xi_{\T^{d}_{n}}(\sigma) = - \sum_{u, v \in \T^{d}_{n}, u \sim v} \Ind{\sigma_u = \sigma_v},
\end{align}
where $\xi : \partial \T^{d}_{n} \to [q]$ is the boundary condition for $\sigma$ i.e. $\sigma_{\vert \partial \T^{d}_{n}} = \xi$.
We define the Gibbs measure as 
\begin{align}
    \mu^\xi_{\T^{d}_{n}, \beta, q}[\sigma] = \frac{1}{Z^\xi_{ \T^{d}_{n}, \beta, q}} \exp\Ll(- \beta H^\xi_{\T^{d}_{n}}(\sigma) \Rr),
\end{align}
where the partition function is defined as 
\begin{align}
\label{e.af}
    Z^\xi_{\T^{d}_{n},\beta, q} = \sum_{\sigma \in [q]^{\T^{d}_{n}},  \sigma_{\vert \partial \T^{d}_{n}} = \xi}\exp\Ll(- \beta H^\xi_{\T^{d}_{n}}(\sigma) \Rr).
\end{align}
When $\beta > 0$, it is the ferromagnetic Potts model where the adjacent spins tend to be equal, and when $\beta < 0$, it is the antiferromagnetic Potts model (adjacent spins tend to be different) and it is the object we would like to study.

It is clear, that the probability of each spin configuration in the AF Potts model only depends on the number of monochromatic edges, i.e., edges that connecting two identical spins. Thus by setting
\begin{align}\label{eq.Weight}
    p := e^\beta \in (0,1),   \qquad  W^\xi_{n,p,q}(\sigma) := p^{\#\{\{u,v \}: u,v \in \T^d_n,\, u \sim v, \, \sigma_u = \sigma_v\}} =\exp\Ll(- \beta H^\xi_{\T^{d}_{n}}(\sigma) \Rr),
\end{align}
and using $\mu^\xi_{n, p, q}$ short for $\mu^\xi_{\T^{d}_{n},\beta, q}$, $Z^\xi_{n, p, q}$ short for $Z^\xi_{\T^{d}_{n},\beta, q}$, we obtain that 
\begin{align*}
Z^\xi_{n,p, q} &= \sum_{\sigma \in [q]^{\T^{d}_{n}},  \sigma_{\vert \partial \T^{d}_{n}} = \xi} W^\xi_{n,p,q}(\sigma),\\
\mu^\xi_{n,p,q}[\sigma] &= W^\xi_{n,p,q}(\sigma)/Z^\xi_{n, p, q}.
\end{align*}

We use $*$ to indicate the root and use the dictionary notation $\{1,2,3, \cdots, d\}^{\N^+}$ for other vertices on the tree. One of the most well-studied problem in statistical physics is the so-called uniqueness and non-uniqueness phase transition of the Gibbs measure on (infinite) graphs. Roughly, the system has a unique Gibbs state if the spin at each vertex is ``not too sensitive'' to the value of the spins at long distances \cite{Dob}. In the case of regular trees, it is shown by \cite{BW} that the uniqueness of Gibbs state is equivalent to the maximal discrepancy of the probability distribution at the root tends to zero, namely,
\begin{align}
\label{e.unique}
    \limsup_{n \to \infty} \max_{\xi : \partial \T^{d}_{n} \to [q]} \mu^\xi_{n,p,q}[\sigma_* = 1] = \frac{1}{q}.
\end{align}

Dobrushin \cite{Dob} formulated an explicit condition that guarantees the uniqueness of Gibbs states, which implies the AF $q$-state Potts model satisfies the uniqueness condition whenever either $\beta\leq  \frac{C_q}{\Delta}$, or $q>2\Delta$, where 
$\Delta$ is the maximal degree of the graph (see also \cite{SSabs}). These conditions are far from sharp, and one expects that for the AF $q$-state Potts model on a regular lattice $\Zd$, there is a critical temperature $\beta_c(d)$, such that the model has a unique Gibbs state for all $\beta < \beta_c$, and exhibit multiple Gibbs states for $\beta>\beta_c$. Very few is known for the AF $q$-state Potts model on $\Zd$: there is some improvement for the Dobrushin uniqueness condition on $\mathbb{Z}^2$ (see, e.g., \cite{GJMP} for AF Potts model and~\cite{BD2012} for ferromagnetic Potts model); proof of the existence of the multiple Gibbs states in very high dimension and very low temperature \cite{Pe, FS, PS}; and proof of the uniqueness of Gibbs state for $3$-coloring on $\mathbb{Z}^2$ \cite{RS} (based on the result of the height function representation in \cite{DCetc}). Otherwise, the conjecture remains largely open on $\Zd$. 

In this paper, we focus on the AF Potts model on regular trees and study its critical exponent. We assume the following hypothesis of the uniqueness of Gibbs states on trees throughout the paper. Notice that for our main result, Theorem \ref{thm.main},  we assume the following hypothesis with the critical value $p= 1 - \frac{q}{d+1}$. 

\begin{hypothesis}\label{hyp}
Let $q, d \in \N^+$ and $q \geq 2$. Also let $p  \in [1 - \frac{q}{d+1}, 1) \cap (0,1)$. Assume that the $q$-state AF Potts model on $d$-ary tree with parameter $p$, defined via the measure \eqref{eq.Weight}, has a unique Gibbs state in the sense of \eqref{e.unique}.
\end{hypothesis}

This uniqueness/non-uniqueness assumption is confirmed in several cases. In particular, the uniqueness/non-uniqueness threshold for proper $q$-colorings (i.e., the $\beta= -\infty$ case in \eqref{e.af}) was proved by Jonasson~\cite{Jon}, building upon the work of Brightwell and Winkler~\cite{BW}. Jonasson showed that the model on the $d$-ary tree exhibits uniqueness if $q > d+1$. When $q \leq d + 1$, non-uniqueness follows from the existence of so-called ``frozen'' colorings and semitranslation-invariant Gibbs measures~\cite{BW}. For the general AF Potts model, it has been known since the 80s that non-uniqueness holds when $p< p_c$, see~\cite{PLM83}. For fixed small values of $q$, it is proved that uniqueness holds on the $d$-ary tree when $p \geq p_c$ and $p > 0$ (see~\cite{GGY2018} for the $q = 3$, $d\geq 2$ case and the $q=d=4$ case, and see~\cite{BBR2020} for the $q = 4$, $d\geq 4$ case). A recent work \cite{bencs2022uniqueness} also claims the uniqueness for $q \geq 5$, $p \geq p_c(q,d)$ with very large $d$.

Our main results is based on Hypothesis~\ref{hyp}. Like many statistical physics model, at the critical temperature, the AF Potts model on the $d$-ary tree also illustrates a power law decay, and the critical exponent is universal for the color number $q$ and the degree $d$. 

\begin{theorem}[Power law decay]\label{thm.main}
Under Hypothesis~\ref{hyp} and for the critical case $p_c = 1 -  \frac{q}{d+1} > 0$, the convergence of marginal probability follows the power law that 
 \begin{align}\label{eq.power}
\lim_{n \to \infty} \frac{1}{n} \Ll(\max_{\xi : \partial \T^{d}_{n} \to [q]}  \Ll\vert \mu^\xi_{n,p_c,q}[\sigma_* = 1] - \frac{1}{q}\Rr\vert\Rr)^{-2} = \frac{d^2-1}{6d^2}\Ll(\frac{q^2}{q-1}\Rr)^2.    
 \end{align}
 \end{theorem}

 A direct corollary of the main theorem is a upper and a lower bound for the convergence rate at critical temperature: there exist $0 < c(d,q) < C(d,q) < \infty$ such that for $p_c = 1 -  \frac{q}{d+1} > 0$ 
 \begin{align}\label{eq.CriticalBound}
 \frac{c}{\sqrt{n}} \leq \max_{\xi} \Ll\vert \mu^\xi_{n,p_c,q}[\sigma_* = 1] - \frac{1}{q}\Rr\vert \leq \frac{C}{\sqrt{n}}.     
 \end{align}

 \begin{remark}
 As mentioned above, Hypothesis \ref{hyp} is justified for certain values of $(q,d)$.  Therefore we obtained \eqref{eq.power}  and \eqref{eq.CriticalBound} for the case
 \begin{itemize}
     \item $q=3, 4$, for all $d \geq q$,  and $p_c = 1 -  \frac{q}{d+1}$;
     \item for any $q \geq 5$, $ d\ge d_0(q)$ sufficiently large, and $p_c = 1 -  \frac{q}{d+1}$.
 \end{itemize}
   We also remark that the critical exponent is only valid for $d \geq 2$, because for $d=1$ it is either $q=1$ a trivial case or $q=2$ for two-coloring problem with long range order.
 \end{remark}

 \begin{remark}
 Using a similar, and somewhat simpler argument, under  Hypothesis \ref{hyp}, we also obtain the exact exponential decay rate for the two point function in the sub-critical phase, namely 
  \begin{align}\label{eq.expo}
    \lim_{n \to \infty} \frac{1}{n}\log \Ll(\max_{\xi : \partial \T^{d}_{n} \to [q]}   \Ll\vert \mu^\xi_{n,p,q}[\sigma_* = 1] - \frac{1}{q}\Rr\vert\Rr) = \log \Ll(\frac{d(1-p)}{p+q-1}\Rr).
 \end{align}
 Again this is verified for $q=3, 4$, for all $d \geq q$,  and $p>p_c$; and for any $q \geq 5$, $ d\ge d_0(q)$ sufficiently large, and $p>p_c$.
 \end{remark}
 
 \bigskip

 We now describe the strategy to prove our main result. Let $\sigma^u$ be the subset of spins on the tree rooted by $u \in \{1,2,\cdots,d\}$ and $\xi^u$ the boundary condition on this subtree. Then one can obtain a recurrence equation
\begin{align}\label{eq.Recurrence}
\mu^{\xi}_{n,p,q}[\sigma_* = 1] = \frac{ \prod_{u = 1}^d  \Ll((p-1)\mu^{\xi^u}_{n-1,p,q}[\sigma_u = 1]+1\Rr)}{\sum_{i=1}^q\prod_{u = 1}^d \Ll((p-1)\mu^{\xi^u}_{n-1,p,q}[\sigma_u = i]+1\Rr)}.    
\end{align}
A useful technique is to consider the ratio in order to simplify the expression  
\begin{equation}\label{eq.ratio1}
\begin{split}
r_n(\xi) &:= \frac{\mu^\xi_{n,p,q}[\sigma_* = 2]}{\mu^\xi_{n,p,q}[\sigma_* = 1]} = \prod_{u = 1}^d \Ll( \frac{(p-1)\mu^{\xi^u}_{n-1,p,q}[\sigma_u = 2]+1}{(p-1)\mu^{\xi^u}_{n-1,p,q}[\sigma_u = 1]+1} \Rr), \\
r^*_n &:= \max_{\xi} r_n(\xi), 
\end{split}
\end{equation}
and then prove the maximal ratio $r^*_n$ converges to $1$ at a power law rate. 

Notice that the ratio $r_n(\xi)$ depends on the boundary condition, which changes with respect to $n$ in \eqref{eq.Recurrence}. Only for the boundary condition that is a pure state of one color, the expression can be closed as an iteration. We call this situation \emph{pure boundary condition} and we use $\mu^i_{n,p,q}$ for the Gibbs measure with pure boundary condition $\xi \equiv i \in [q]$. Afterwards, We define a quantity  
\begin{align}\label{eq.ratio2}
    r_n :=  \frac{\mu^2_{n,p,q}[\sigma_* = 2]}{\mu^2_{n,p,q}[\sigma_* = 1]},
\end{align}
and the sequence $(r_n)_{n \in \N}$ satisfies an iteration
\begin{align}\label{eq.f}
r_{n+1} = f(r_n), \qquad f(x) = \Ll(\frac{p x + (q-1)}{p + (q-2) +x}\Rr)^d.
\end{align}
The iteration function $f$ here plays an important role in the analysis of convergence. By an heuristic asymptotic analysis, we have 
\begin{align}\label{eq.contraction}
f(x) - 1 = \Ll(1 + \frac{(1-p)(1-x)}{p + (q-2) +x}\Rr)^d - 1 \simeq  \frac{d(1-p)(1-x)}{p + (q-2) +x}.
\end{align}
This implies that $x=1$ is a fixed point, and the uniqueness regime corresponds to the contraction rate less than $1$ near the fixed point that 
\begin{align}\label{eq.pc}
\frac{d(1-p)}{p + q - 1} \leq 1   \Longleftrightarrow  p \geq 1 - \frac{q}{d+1}.
\end{align}
By analysis of the two-step iteration function $(f \circ f)$, we obtain $r_{2n} \simeq 1 + \frac{C}{\sqrt{2n}}$. The critical exponent $\frac{1}{2}$ and the constant $\frac{d^2-1}{6d^2}$ in \eqref{eq.power} are closely related to the analytical expansion of $(f \circ f)$ near $1$.

The above reasoning gives the critical exponent for the pure boundary conditions. A natural question to ask, is whether the pure boundary condition dominates other boundary conditions. This is unfortunately false. For the ferromagnetic Ising model, it is known that the magnetization is an increasing function of the boundary condition, thanks to the GKS inequality \cite{GKS}. For the ferromagnetic Potts model, the GKS inequality no longer applies, yet a weaker form of the positive association can be derived via the random cluster representation with the FKG inequality \cite{Gri}. Such classical representation and FKG inequality are in general false for the AF Potts model. Some correlation inequalities, such as the first Griffith's inequality, were proved for the AF Potts model on bipartite graphs \cite{ferreira1999antiferromagnetic}, based on the Swendsen-Wang-Kotecky algorithm, but they are not sufficient to deduce the boundary domination. 

To explain better the non-existence of boundary condition domination, we give the following example in Figure~\ref{fig.Frozen}. It is an AF Potts model with $q = 3$ and $d = 3$, then for any $p \in (0,1)$ it is always in the subcritical regime. For every fixed boundary condition, there are $3^3=27$ configurations. The boundary condition for the example on the left-hand side is the one of pure boundary condition, but we see that this configuration with $\sigma_* = 3$ also has no monochromatic edges, which implies that $\mu^1_{2,p,q}[\sigma_* = 3] = \mu^1_{2,p,q}[\sigma_* = 2] \geq \frac{1}{27}$ and $\mu^1_{2,p,q}[\sigma_* = 1] \leq \frac{25}{27}$. On the other hand, the configuration on the right-hand side is the one with so called frozen boundary $\xi'$. If $p = 0$, then it is the configuration which charges all the weight because the other configurations have at least one monochromatic edge. When $p > 0$, we have $\mu^{\xi'}_{2,p,q}[\sigma_* = 1] \geq \frac{1}{1 + 26 p}$, so for $p$ very close to $0$, this probability is larger than $\mu^1_{2,p,q}[\sigma_* = 1]$. This example can be generalized to any $n$ which justifies that the pure boundary is not the maximiser for the marginal probability at the root.
\begin{figure}[h!]
    \centering
    \includegraphics[scale = 0.5]{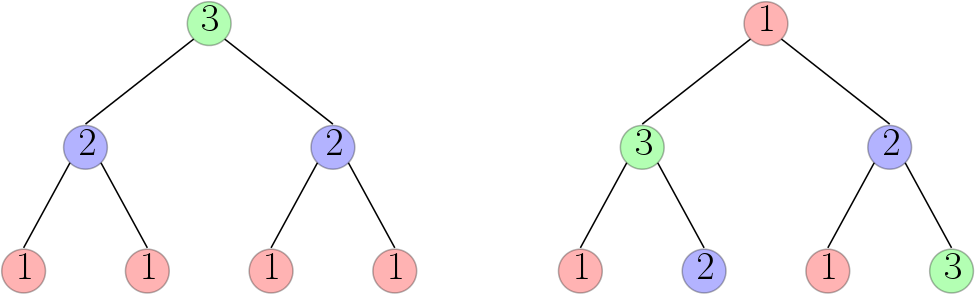}
    \caption{The example on the left-hand side is the one with the pure boundary condition, while the one on the right-hand side is with the frozen boundary.}
    \label{fig.Frozen}
\end{figure}

Despite the lack of boundary condition domination, it is also not clear which boundary condition realizes $r^*_n$. Thus, the analysis of $(r_n)_{n \in \N}$ above  only gives a lower bound of the discrepancy probability. However, we will prove that 
\begin{align}\label{eq.ratioDomination}
\text{ for all } r^*_{n} = 1 + o(1), \qquad   \text{we have } r^*_{n+2} \leq (f \circ f)(r^*_{n}).
\end{align}
In another word, when $r^*_{n}$ is very close to $1$, its upper bound is also dominated by the two-step iteration function $(f \circ f)$. Therefore, the upper bound and lower bound follow the same dynamic iteration near the fixed point, which implies Theorem~\ref{thm.main}.

Let us add more comments on \eqref{eq.ratioDomination} and the iteration. The upper bound iteration \eqref{eq.ratioDomination} here is obtained quite qualitatively, i.e. we do not go further to give all possible values of $r_n^*$ such that $r^*_{n+2} \leq (f \circ f)(r^*_{n})$ holds. That is why we need to assume Hypothesis~\ref{hyp}, which guarantees $r_n^*= 1+o(1)$ for large $n$. On the other hand, we obtained precise contraction rate of the two-step iteration $(f \circ f)$ and its generalization. Therefore, if we could enlarge the regime such that \eqref{eq.ratioDomination} holds, and show that $r_n^*$ enters this regime for large $n$, it will give a unconditional proof for Theorem \ref{thm.main}.     
 
\bigskip 

The rest of the paper is organized as following. In Section 2, we prove the lower bound of $(r^*_n)_{n \in \N}$ by studying the pure boundary iterations. In Section 3, we formulate the general two-step iterations, and justify the asymptotic upper bound \eqref{eq.ratioDomination}, when $r_n^*$ lies in a small neighborhood of $1$. In Section 4, we give a proof of the contraction rate of the $(f \circ f)$ type maps, thus concluding Theorem \ref{thm.main}, and is of independent interest. 

\subsection*{Notations}
Throughout the paper, we will use the notation of Dirac function $\delta_{i}(j) = \delta_{i,j} = \Ind{i=j}$. We also recall that we use the dictionary notation $\{1,2,3, \cdots, d\}^{\N^+}$ for the vetrices of the tree, and $\vert u \vert$ for the word length. Then for $u \in \{1,2,3, \cdots, d\}^{\N^+}$ with $\vert u \vert = k$, we use $\xi^u$ to represent the boundary condition of the subtree rooted at $u$, and  $\mu^{\xi^u}_{n-k, p, q}, W^{\xi^u}_{n-k,p,q}$ respectively for its probability space  and the weight function.

\section{Lower bound by pure boundary condition}\label{sec.pure}
In this section, we prove the lower bound in Theorem~\ref{thm.main}. Our method is to establish the one-step iteration, which has a closed expression for the pure boundary condition case, i.e. $r_{n+1} = f(r_n)$ in Lemma~\ref{lem.pure}. Therefore, we will focus on the iteration function $f$ defined in \eqref{eq.f}.

\subsection{One-step iteration}
In this part, we establish the recurrence called \emph{one-step iteration}. 
\begin{lemma}\label{lem.recurrence}
For any $p \in (0,1)$, integers $d, n \geq 1, q \geq 2$, we have the following recurrence for the ratio function
\begin{align}\label{eq.ratio1New}
    \frac{\mu^\xi_{n,p,q}[\sigma_* = 2]}{\mu^\xi_{n,p,q}[\sigma_* = 1]} = \prod_{u = 1}^d \Ll( \frac{(p-1)\mu^{\xi^u}_{n-1,p,q}[\sigma_u = 2]+1}{(p-1)\mu^{\xi^u}_{n-1,p,q}[\sigma_u = 1]+1} \Rr).
\end{align}
\end{lemma}
Similar recursions are also used in \cite{GGY2018, BBR2020}, and the proof are almost the same. However, for completeness we present the proof as follows.
\begin{proof}
Using the weight function defined in \eqref{eq.Weight}
\begin{align*}
W^{\xi}_{n,p,q} (\sigma_* =1)=  \prod_{u = 1}^d \Ll(W^{\xi^u}_{n-1,p,q}(\sigma^u ) p^{\delta_1(\sigma_u)}\Rr),
\end{align*}
which implies
\begin{align}
 \mu^\xi_{n,p,q}[\sigma_* = 1] = \frac{\sum_{\sigma^1, \cdots, \sigma^d} \prod_{u = 1}^d \Ll(W^{\xi^u}_{n-1,p,q}(\sigma^u ) p^{\delta_1(\sigma_u)}\Rr)}{\sum_{j=1}^q\sum_{\sigma^1, \cdots, \sigma^d} \prod_{u = 1}^d \Ll(W^{\xi^u}_{n-1,p,q}(\sigma^u ) p^{\delta_j(\sigma_u)}\Rr)}.    
\end{align}
 Then we notice that conditioned on the state of $\sigma_*$, the spins on different subtrees are independent, so we have 
\begin{align}\label{eq.Recurrence0}
 \mu^\xi_{n,p,q}[\sigma_* = 1] = \frac{ \prod_{u = 1}^d \Ll(\sum_{\sigma^u} W^{\xi^u}_{n-1,p,q}(\sigma^u ) p^{\delta_1(\sigma_u)}\Rr)}{\sum_{j=1}^q \prod_{u = 1}^d \Ll(\sum_{\sigma^u} W^{\xi^u}_{n-1,p,q}(\sigma^u ) p^{\delta_j(\sigma_u)}\Rr)}. 
\end{align}
After a normalization by $Z^{\xi^u}_{n-1,\beta, q}$, this equation becomes
\begin{align*}
\mu^{\xi}_{n,p,q}[\sigma_* = 1] = \frac{ \prod_{u = 1}^d  \Ll(\sum_{j_u \in [q]}\mu^{\xi^u}_{n-1,p,q}[\sigma_u = j_u ] p^{\delta_1(j_u)}\Rr)}{\sum_{i=1}^q\prod_{u = 1}^d \Ll(\sum_{j_u \in [q]}\mu^{\xi^u}_{n-1,p,q}[\sigma_u = j_u ] p^{\delta_i(j_u)}\Rr)},
\end{align*}
and this can be further simplified by 
\begin{align*}
\sum_{j_u \in [q]}\mu^{\xi^u}_{n-1,p,q}[\sigma_u = j_u ] p^{\delta_i(j_u)} &= p\mu^{\xi^u}_{n-1,p,q}[\sigma_u = i] + \sum_{j_u \in [q] \setminus \{i\}}\mu^{\xi^u}_{n-1,p,q}[\sigma_u = j_u ] \\
&=  (p-1)\mu^{\xi^u}_{n-1,p,q}[\sigma_u = i]+1.    
\end{align*}
Therefore, we have
\begin{align*}
\mu^{\xi}_{n,p,q}[\sigma_* = 1] = \frac{ \prod_{u = 1}^d  \Ll((p-1)\mu^{\xi^u}_{n-1,p,q}[\sigma_u = 1]+1\Rr)}{\sum_{i=1}^q\prod_{u = 1}^d \Ll((p-1)\mu^{\xi^u}_{n-1,p,q}[\sigma_u = i]+1\Rr)}.    
\end{align*}
This gives us the desired result.
\end{proof}

Generally speaking, the one-step iteration is not a closed formula for $r_{n}(\xi)$, because as we move to different subtrees,  the boundary conditions $\xi^u$ are different from $\xi$. A very special case is the pure boundary situation, in which we obtain a closed formula.
\begin{lemma}\label{lem.pure}
For any $p \in (0,1)$, integers $d \geq 1, q \geq 2$ and the pure boundary condition $\xi \equiv 2$,  the ratio function ${r_n = \frac{\mu^2_{n,p,q}[\sigma_* = 2]}{\mu^2_{n,p,q}[\sigma_* = 1]}}$ satisfies the iteration $r_{n+1} = f(r_n)$ with $f$ defined in \eqref{eq.f}. 
\end{lemma}
\begin{proof}
When the boundary condition is pure $\xi \equiv 2$, the probability $\mu^2_{n,p,q}[\sigma_* = i]$ only takes two values, and we can calculate that 
\begin{align}\label{eq.pureProba}
\mu^2_{n,p,q}[\sigma_* = 2] = \frac{r_n}{r_n + q - 1}, \qquad \mu^2_{n,p,q}[\sigma_* = 1] = \frac{1}{r_n + q - 1}. 
\end{align}
We put them back to \eqref{eq.ratio1New} and obtain the iteration function $f$.
\end{proof}

\subsection{Convergence rate for the pure boundary condition}
The main result in this part is the following convergence rate in the special case of pure boundary condition.
\begin{proposition}\label{prop.pure}
For the critical case $p_c = 1 -  \frac{q}{d+1} > 0$, the convergence rate for the pure boundary condition is that 
 \begin{align}\label{eq.pure}
\lim_{n \to \infty} \frac{1}{n} \Ll(  \Ll \vert \mu^{1}_{n,p_c,q}[\sigma_* = 1] - \frac{1}{q} \Rr \vert \Rr)^{-2} = \frac{d^2-1}{6d^2}\Ll(\frac{q^2}{q-1}\Rr)^2.    
 \end{align}
 \end{proposition}
We remark that for the pure boundary condition, the convergence of the ratio $\lim_{n \to \infty} r_n = 1$ follows directly from explicit computations in Proposition~\ref{prop.TwoStepfm} and Corollary~\ref{cor.rn}. Therefore the proposition above does not require assuming Hypothesis~\ref{hyp}.

The following elementary properties about $f$ are useful. 
\begin{lemma}[Elementary properties of $f$]\label{lem.Element}
For $p  \in [1 - \frac{q}{d+1}, 1) \cap (0,1)$, the following properties hold for $f$ defined in \eqref{eq.f}.
\begin{enumerate}
    \item Range: $f(x) \geq 1$ for all $x \in [0,1]$; $f(x) \in (0,1]$ for all $x \geq 1$.
    \item Fixed point: $f(x) = x$ admits a unique solution on $\R^+$ that $x=1$.
    \item Monotonicity: $f$ is decreasing.
    \item Contraction on one side: for every $x \geq 1$, $1 - f(x) \leq x - 1$.
\end{enumerate}
\end{lemma}
\begin{proof}
Property (1) can be deduced directly from the expression of $f$ and the fixed point is a corollary. To study the monotonicity of $f$, we look at its derivative 
\begin{align}\label{eq.Df}
f'(x) = \Ll(\frac{p x + (q-1)}{p + (q-2) +x}\Rr)^{d-1} \frac{d(p-1)(p+q-1)}{(p + (q-2) +x)^2} \leq 0,
\end{align}
because we have $p < 1$. For the contraction on one side, it suffices to study ${F(x) = f(x) + x - 2}$ and we know that $F(1) = 0$. Thus we look at its derivative. It is clear that 
\begin{align}\label{eq.Ff}
\forall x \geq 1, \qquad F'(x) = f'(x) + 1 \geq 0,    
\end{align}
which can be deduced directly from \eqref{eq.Df}.
\end{proof}

A direct corollary from Lemma~\ref{lem.Element} is that $r_{2k} \geq 1$ while $r_{2k+1} \leq 1$ for all $k \in \N$, so the ratio converges to the fixed point from two sides. This suggests us to use the two-step iteration $(f \circ f)$ to study the convergence rate for even terms and odd terms separately.

\begin{proof}[Proof of Proposition~\ref{prop.pure}]
Throughout the proof we admit $\lim_{n \to \infty} r_n = 1$, which will be proved in Corollary~\ref{cor.rn}. We divide the proof into three steps.

\emph{Step 1: identification of the critical exponent.} We explain at first the intuition to find the critical exponent. Suppose that $r_{2k} \simeq 1 + (C k)^{-\frac{1}{\alpha}}$, then we should have
\begin{align*}
(r_{2k} - 1)^{-\alpha} \simeq C k,    
\end{align*}
which implies
\begin{align*}
(r_{2k+2} - 1)^{-\alpha} - (r_{2k} - 1)^{-\alpha} \simeq C.    
\end{align*}
Thus, the problem can be reduced to find an exponent $\alpha > 0$ and a non-trivial  constant $C > 0$ in order to see the power law behavior
\begin{align}\label{eq.critical}
\lim_{x \searrow 1}((f \circ f)(x) - 1)^{-\alpha} - (x - 1)^{-\alpha} = C.    
\end{align}
Since $f$ is analytic, we do the Taylor expansion for $(f \circ f)$ around $1$ that 
\begin{align}
(f \circ f)(x) - 1 = \sum_{m=1}^{\infty} \frac{c_m}{m!}  (x-1)^m, \qquad c_m = \frac{\d^m }{\d x^m}_{\vert_{x=1}} (f \circ f)(x). 
\end{align}
At the critical phase, one can calculate several terms using \eqref{eq.Df} and the chain rule that 
\begin{align*}
c_1 = (f \circ f)'(1) = 1, \qquad
&c_2 = (f \circ f)''(1) = 0, \qquad
c_3 = (f \circ f)'''(1) = -\frac{(d^2-1)}{d^2}. 
\end{align*}
We explain more details. For $c_1$ and $c_2$, we only need to use the fact $f(1) = 1$ and $f'(1) = -1$ by \eqref{eq.f}, \eqref{eq.Df} and $p_c = 1 - \frac{q}{d+1}$. Then by the chain rule 
\begin{align*}
(f \circ f)'(1) &= f'(f(1))f'(1) = 1,  \\
(f \circ f)''(1) &= f''(1)(f'(1))^2 + f'(1)f''(1) = 0.
\end{align*}
Notice that $c_2$ is null, and since we assume $d \geq 2$, the term $c_3$ does not vanish. These terms will play an important role in the following calculation. 

We put the analytic expansion back in \eqref{eq.critical} with a change of variable $y := x-1$
\begin{align*}
\lim_{x \searrow 1}((f \circ f)(x) - 1)^{-\alpha} - (x - 1)^{-\alpha} = \lim_{y \searrow 0}\frac{y^{\alpha} - (y+\frac{c_3}{6} y^3 + o(y^3))^{\alpha}}{y^\alpha(y+\frac{c_3}{6} y^3 + o(y^3))^{\alpha}}.
\end{align*}
Here we use the notation $o(1)$ for $\lim_{y \to 0} o(1) = 0$. Then we observe that the numerator is of leading order $y^{2+\alpha}$
\begin{align*}
y^{\alpha} - \Ll(y+\frac{c_3}{6} y^3 + o(y^3)\Rr)^{\alpha} = y^{\alpha} - y^{\alpha}\Ll(1 +\frac{c_3}{6} y^2 + o(y^2)\Rr)^\alpha = -\frac{\alpha c_3 }{6} y^{2+\alpha} + o(y^{2+\alpha}),  
\end{align*}
while the denominator is of leading order $y^{2\alpha}$. Therefore,  $\alpha = 2$ is the only possible candidate to realize a non-trivial limit 
\begin{equation}\label{eq.ffLimit}
\begin{split}
\lim_{x \searrow 1}((f \circ f)(x) - 1)^{-2} - (x - 1)^{-2} &= \lim_{y \searrow 0}\frac{y^{2} - (y+\frac{c_3}{6} y^3 + o(y^3))^{2}}{y^2(y+\frac{c_3}{6} y^3 + o(y^3))^{2}}\\
&= \lim_{y \searrow 0}\frac{-\frac{c_3}{3} y^4 + o(y^4))}{(y^4+\frac{c_3}{3} y^6 + o(y^6))}\\
&= - \frac{c_3}{3}.    
\end{split}
\end{equation}

\medskip

\emph{Step 2: convergence rate of $(r_{n})_{n \in \N}$.}
By Proposition~\ref{prop.TwoStepfm}, we have $\lim_{k \to \infty}r_{2k} = 1$ and then use the telescope formula to obtain that 
\begin{align*}
\lim_{n \to \infty} \frac{1}{2n+2} (r_{2n+2}-1)^{-2}
&= \lim_{n \to \infty} \frac{1}{2n+2} \Ll(\sum_{k=1}^{n} \Ll((r_{2k+2}-1)^{-2} - (r_{2k}-1)^{-2}\Rr)\Rr) \\
&= \lim_{k \to \infty} \frac{1}{2}\Ll((r_{2k+2}-1)^{-2} - (r_{2k}-1)^{-2}\Rr)\\
&= - \frac{c_3}{6} = \frac{d^2-1}{6d^2}.
\end{align*}
From the first line to the second line, we use a classical exercise that the limit of average coincides with the limit of sequence. This gives us the ratio convergence rate for the even terms.

We then treat the odd terms with the mean value theorem: there exists  $z_{2k} \in (1, r_{2k})$ such that
\begin{align*}
r_{2k+1} - 1 = f(r_{2k}) - f(1) = f'(z_{2k})(r_{2k} - 1), \end{align*}
and $\lim_{k \to \infty}z_{2k} = 1$ by squeeze theorem. This implies $\lim_{k \to \infty} f'(z_{2k}) = -1$ and thus we establish that 
\begin{align}\label{eq.ratioRate1}
    \lim_{n \to \infty} \frac{1}{n}\vert  r_n - 1 \vert^{-2} = \frac{d^2-1}{6d^2}.
\end{align}

\medskip
\emph{Step 3: convergence rate of probability.} We now turn to the proof of \eqref{eq.pure}. We make use of the formula \eqref{eq.pureProba}, Step 2 (which gives $r_n \to 1$) and \eqref{eq.ratioRate1}
\begin{align*}
\lim_{n \to \infty} \frac{1}{n} \Ll(  \Ll \vert \mu^{2}_{n,p_c,q}[\sigma_* = 2] - \frac{1}{q} \Rr \vert \Rr)^{-2} &= \lim_{n \to \infty} \frac{1}{n} \Ll(  \Ll \vert \frac{r_n}{r_n + q - 1} - \frac{1}{q} \Rr \vert \Rr)^{-2}\\
&= \lim_{n \to \infty} \frac{1}{n} \Ll(  \Ll \vert \frac{(q-1)\vert r_n - 1 \vert}{q (r_n + q - 1)}  \Rr \vert \Rr)^{-2}\\
&= \Ll(\frac{q^2}{q-1}\Rr)^2  \lim_{n \to \infty}\frac{1}{n}\vert r_n - 1 \vert^{-2} \\
&= \frac{d^2-1}{6d^2} \Ll(\frac{q^2}{q-1}\Rr)^2 . 
\end{align*}
This is the desired result.
\end{proof}

\section{Upper bound by expansion near fixed point}
The pure boundary condition analyzed in the previous section is only a very special case. In this section we show that for evey boundary condition,  the iteration function $(f \circ f)$, where $f$ is defined in \eqref{eq.f} captures the rate of convergence near the fixed point. The main result of this section is the following proposition.
\begin{proposition}\label{prop.TwoStep}
For any integers $d,n \geq 1, q \geq 2$ and $p  \in [1 - \frac{q}{d+1}, 1) \cap (0,1)$, there exists a constant $\epsilon(d,q,p) \in (0, \infty)$, such that for every $r^*_n \in [1, 1 + \epsilon)$, we have that ${r^*_{n+2} \leq (f \circ f)(r^*_n)}$.
\end{proposition}

\subsection{Two-step iteration}
We may obtain a two-step iteration by applying Lemma~\ref{lem.recurrence} twice. However, two step iteration in the general boundary condition is much more complicated than the one studied in the last section, so we will make a reduction to the extremal boundary conditions, in the sense of the next proposition.  

\begin{proposition}\label{prop.TwoStepBound}
For any integers $d,n \geq 1, q \geq 2$ and $p  \in [1 - \frac{q}{d+1}, 1) \cap (0,1)$, the ratio function satisfies the following bound for the two-step iteration
\begin{align}\label{eq.TwoStepBound}
    r^*_{n+2} \leq \max_{\A(r^*_n)} h(\x).
\end{align}
Here $\x = \{x^u_k\}_{u \in [d], k \in [q] }$ represents a vector, and $h$ is a function defined as 
\begin{equation}\label{eq.h}
\begin{split}
h(\x) := U^d,    & \qquad U := \frac{\sum_{j=1}^q V_j p^{\delta_{2}(j)}}{\sum_{j=1}^q V_j p^{\delta_{1}(j)}},\\
V_j := \prod_{u=1}^d V^u_j, & \qquad V^u_j := \frac{\sum_{k=1}^q x^{u}_k p^{\delta_{j}(k)}}{\sum_{k=1}^q x^{u}_k p^{\delta_{1}(k)}},
\end{split}
\end{equation}
and  $\A(r)$ as a reduced admissible domain for $r \geq 1$ 
\begin{align}\label{eq.Ar}
\A(r) := \Bigg\{\x : x^u_1 = 1 \text{ and } x^u_k \in \{1, r\},  \forall k \in [q] \setminus \{1\}, \text{ and } \forall u \in [d]\Bigg\}.
\end{align}
\end{proposition}

This proposition has already been shown in \cite[Lemma 2.6]{GGY2018}. We reformulate its proof here for the self-completeness of the paper.
\begin{proof}[Proof of Proposition~\ref{prop.TwoStepBound}] \emph{Step 1: general two-step iteration formula.} Using the recurrence \eqref{eq.Recurrence0} formula, we have
\begin{align*}
r_{n+2}(\xi) &=  \frac{W^\xi_{n+2, p, q}(\sigma_* = 2)}{W^\xi_{n+2, p, q }(\sigma_* = 1)}\\
&=  \prod_{v=1}^d \Ll(\frac{\sum_{j=1}^q W^{\xi^v}_{n+1, p, q}(\sigma_v = j) p^{\delta_{2}(j)}}{\sum_{j=1}^q W^{\xi^v}_{n+1, p, q}(\sigma_v = j) p^{\delta_{1}(j)}}\Rr) \\
&=  \prod_{v=1}^d \Ll(\frac{\sum_{j=1}^q W^{\xi^v}_{n+1, p, q}(\sigma_v = j)/W^{\xi^v}_{n+1, p, q}(\sigma_v = 1) p^{\delta_{2}(j)}}{\sum_{j=1}^q W^{\xi^v}_{n+1, p, q}(\sigma_v = j)/W^{\xi^v}_{n+1, p, q}(\sigma_v = 1) p^{\delta_{1}(j)}}\Rr).
\end{align*}
We can apply once again the formula \eqref{eq.Recurrence0} for $W^{\xi^v}_{n+1, p, q}(\sigma_v = j)/W^{\xi^v}_{n+1, p, q}(\sigma_v = 1)$ in order to get the general two-step iteration. We define a function $\hat{h}$
\begin{equation}\label{eq.hNew}
\begin{split}
\hat{h}(\{\hat{x}^{vu}_k\}_{u,v \in [d], k \in [q]}) := \prod_{v=1}^d \hat{U}^v,    &\qquad
\hat{U}^v := \frac{\sum_{j=1}^q \hat{V}^v_j p^{\delta_{2}(j)}}{\sum_{j=1}^q \hat{V}^v_j p^{\delta_{1}(j)}}, \\
\hat{V}^v_j := \prod_{u=1}^d \hat{V}^{vu}_j, & \qquad \hat{V}^{vu}_j := \frac{\sum_{k=1}^q \hat{x}^{vu}_k p^{\delta_{j}(k)}}{\sum_{k=1}^q \hat{x}^{vu}_k p^{\delta_{1}(k)}},
\end{split}
\end{equation}
and then by identifying that 
\begin{equation}\label{eq.TwoStepRole}
\begin{split}
{\hat{U}^v = \frac{\sum_{j=1}^q W^{\xi^v}_{n+1, p, q}(\sigma_v = j) p^{\delta_{2}(j)}}{\sum_{j=1}^q W^{\xi^v}_{n+1, p, q}(\sigma_v = j) p^{\delta_{1}(j)}}}, \qquad & 
\hat{V}^{v}_j = \frac{W^{\xi^{v}}_{n+1, p, q}(\sigma_{v} = j)}{W^{\xi^{v}}_{n+1, p, q}(\sigma_{v} = 1)},\\
{\hat{V}^{vu}_j =  \frac{\sum_{k=1}^q W^{\xi^{vu}}_{n, p, q}(\sigma_{vu} = k) p^{\delta_{j}(k)}}{\sum_{k=1}^q W^{\xi^{vu}}_{n, p, q}(\sigma_{vu} = k) p^{\delta_{1}(k)}}}, \qquad & 
\hat{x}^{vu}_k = \frac{W^{\xi^{vu}}_{n, p, q}(\sigma_{vu} = k)}{W^{\xi^{vu}}_{n, p, q}(\sigma_{vu} = 1)},
\end{split}    
\end{equation}
we obtain that 
\begin{align}\label{eq.TwoStepGeneral}
 r_{n+2}(\xi) =  \hat{h}\Ll(\Ll\{ \frac{W^{\xi^{vu}}_{n, p, q}(\sigma_{vu} = k)}{W^{\xi^{vu}}_{n, p, q}(\sigma_{vu} = 1)}\Rr\}_{u,v \in [d], k \in [q]}\Rr).
\end{align}

The expression of $\hat{h}$ is quite close to that of $h$ defined in \eqref{eq.h}. In the following steps, we see how to transform $\hat{h}$ to $h$ and restrict to the admissible domain. 

\medskip

\emph{Step 2: characterizations of $\max_{\xi}r_{n+2}(\xi)$.} In this part, we give some characterizations of the boundary condition that realizes $\max_{\xi}r_{n+2}(\xi)$. These characterizations are not sufficient to identify the maximiser, but may reduce the maximiser to the admissible domain \eqref{eq.Ar}. We focus on at first $\hat{U}^v$ in \eqref{eq.hNew} and \eqref{eq.TwoStepRole} that
\begin{align}\label{eq.UNew}
\hat{U}^v  =  \frac{\sum_{j=1}^q W^{\xi^v}_{n+1, p, q}(\sigma_v = j) + (p-1) W^{\xi^v}_{n+1, p, q}(\sigma_v = 2) }{\sum_{j=1}^q W^{\xi^v}_{n+1, p, q}(\sigma_v = j) + (p-1) W^{\xi^v}_{n+1, p, q}(\sigma_v = 1) }.
\end{align}
As we are interested in the $\max_{\xi}r_{n+2}(\xi)$, without loss of generality, we can suppose that for the maximiser of $\max_{\xi}r_{n+2}(\xi)$
\begin{multline}\label{eq.Condition1}
\mathbf{E_1}:= \Ll\{\xi : \partial \T^{d}_{n+2} \to [q] \, \Big\vert \, \forall v \in [d], \quad    W^{\xi^v}_{n+1, p, q}(\sigma_v = 1) = \max_{j \in [q]} W^{\xi^v}_{n+1, p, q}(\sigma_v = j), \Rr.\\
\Ll.W^{\xi^v}_{n+1, p, q}(\sigma_v = 2) = \min_{j \in [q]} W^{\xi^v}_{n+1, p, q}(\sigma_v = j)\Rr\}.
\end{multline}
Otherwise, for example suppose that $j \in [q] \setminus \{1\}$ is the color to achieve the maximum, one can exchange the color between $j$ and $1$ in $\xi^v$. Then the value $W^{\xi^v}_{n+1, p, q}(\sigma_v = 1)$ and $W^{\xi^v}_{n+1, p, q}(\sigma_v = j)$ exchange, and $\hat{U}^v$ becomes larger.  

We now go one step further and claim that for the maximiser of $\max_{\xi}r_{n+2}(\xi)$
\begin{align}\label{eq.Condition2}
\mathbf{E_2}:= \Ll\{\xi : \partial \T^{d}_{n+2} \to [q] \, \Big\vert  \, \forall v,u \in [d], \quad    W^{\xi^{vu}}_{n, p, q}(\sigma_{vu} = 1) = \min_{k \in [q]} W^{\xi^{vu}}_{n, p, q}(\sigma_{vu} = k)\Rr\}.     
\end{align}
The argument is as follows. Suppose that it is not true, then there exist $v,u \in[d]$ and  $k \in [q] \setminus \{1\}$, such that $W^{\xi^{vu}}_{n, p, q}(\sigma_{vu} = k)$ achieves the minimum weight on the subtree indexed by $vu$. Then by a similar formula as \eqref{eq.UNew}
\begin{align}\label{eq.VNew}
\hat{V}^{vu}_j  =  \frac{\sum_{k=1}^q W^{\xi^{vu}}_{n, p, q}(\sigma_{vu} = k) + (p-1) W^{\xi^{vu}}_{n, p, q}(\sigma_{vu} = 2) }{\sum_{k=1}^q W^{\xi^{vu}}_{n, p, q}(\sigma_{vu} = k) + (p-1) W^{\xi^{vu}}_{n, p, q}(\sigma_{vu} = 1) },
\end{align}
switching the color between $k$ and $1$ on the boundary condition $\xi^{vu}$ will decrease the value of $\hat{V}^{vu}_j$ for $j \in [q] \setminus \{1\}$, so as that of $\hat{V}^{v}_j$. Viewing the definition of $\hat{V}^{v}_j$ in \eqref{eq.TwoStepRole}, the condition \eqref{eq.Condition1} implies that 
\begin{equation}\label{eq.Condition11}
\begin{split}
\forall v \in [d], j \in [q], \quad    \hat{V}^{v}_j \leq 1 = \hat{V}^{v}_1.
\end{split}
\end{equation}
We then develop $\hat{U}^v$ in function of $\hat{V}^{v}_j$ like in \eqref{eq.h}
\begin{align*}
\hat{U}^{v} =  1 + \frac{(1-p)(\hat{V}^{v}_1 - \hat{V}^{v}_2)}{\sum_{j=1}^q \hat{V}^{v}_j + (p-1)\hat{V}^{v}_1}.    
\end{align*}
Then the decrement of $\hat{V}^{v}_j$ for $j \in [q] \setminus \{1\}$ implies the increment of $\hat{U}^{v}$.

We now propose a third characterization for the maximiser of $\max_{\xi}r_{n+2}(\xi)$
\begin{align}\label{eq.Condition3}
\mathbf{E_3}:= \Ll\{\xi : \partial \T^{d}_{n+2} \to [q] \, \Big\vert  \quad\forall v_1, v_2 \in [d], \quad \xi^{v_1} = \xi^{v_2}\Rr\}. 
\end{align}    
That is to say, the boundary condition for the maximiser is periodic. This comes from the AM--GM inequality, 
\begin{align*}
    \hat{h}\Ll(\Ll\{ \frac{W^{\xi^{vu}}_{n, p, q}(\sigma_{vu} = k)}{W^{\xi^{vu}}_{n, p, q}(\sigma_{vu} = 1)}\Rr\}_{u,v \in [d], k \in [q]}\Rr) = \prod_{v=1}^d \hat{U}^v \leq \frac{1}{d} \sum_{v=1}^d (\hat{U}^v)^d.
\end{align*}
Moreover, the upper bound on {\rhs} is realizable: for a given boundary condition $\xi$, it suffice to take the boundary condition $\xi^v$ of subtrees to maximize $(\hat{U}^v)^d$, and then copy it to other subtrees. Combining \eqref{eq.Condition1}, \eqref{eq.Condition2}, \eqref{eq.Condition3}, we obtain that 

\begin{align}\label{eq.TwoStepChar}
 r^*_{n+2} =  \max_{\xi \in \mathbf{E_1} \cap \mathbf{E_2} \cap \mathbf{E_3} }\hat{h}\Ll(\Ll\{ \frac{W^{\xi^{vu}}_{n, p, q}(\sigma_{vu} = k)}{W^{\xi^{vu}}_{n, p, q}(\sigma_{vu} = 1)}\Rr\}_{u,v \in [d], k \in [q]}\Rr).    
\end{align}

\emph{Step 3: reduction of $\hat{h}$.}
Then the expression of $\hat{h}$ becomes that of $h$ once we skip the index of $v$, and then use \eqref{eq.TwoStepChar}, \eqref{eq.TwoStepRole} to give an upper bound  
\begin{align}\label{eq.TwoStepBound2}
    r^*_{n+2} \leq  \max_{\forall u \in [d], k \in [q], 1 \leq x^u_k \leq r^*_{n}, x^u_1 = 1}h(\x).
\end{align}
Moreover, fixing the other variables, the function $x^u_k \mapsto h(\x)$ is of type $\frac{a x^u_k + b}{c x^u_k +d}$ with $a,b,c,d \in \R$. For this mapping, the extreme value is attained at the endpoint of the interval. Therefore, we obtain the desired result \eqref{eq.TwoStepBound}.
\end{proof}

\subsection{Expansion of the two-step iteration near the fixed point}
In this part, we study the property of the function $h(\x)$ defined in \eqref{eq.h}. Although the optimization problem $\max_{\x \in \A(r)} h(\x)$ is quite difficult for a general $r$, its behavior near the fixed point $1$ is closely related to the function $f$ defined in \eqref{eq.f}. 
\begin{proposition}\label{prop.expansion}
For any integers $d,n \geq 1, q \geq 2$ and $p  \in [1 - \frac{q}{d+1}, 1) \cap (0,1)$, there exists a constant $\epsilon(d,q,p) \in (0, \infty)$, such that for any  $r \in (1,1+\epsilon)$, we have
\begin{align*}
 \max_{\x \in \A(r)} h(\x) = (f \circ f)(r).    
\end{align*}
Moreover, the maximiser is attained when $x^u_k = 1 + (r-1) \delta_{k,2}$ for all $u \in [d]$.
\end{proposition}

\begin{proof}
Throughout the proof, for $\x \in \A(r)$, we set 
\begin{align*}
    x^u_k = 1 + (r-1)\theta^u_k, \qquad \theta^u_k \in \{0,1\},
\end{align*}
with $\theta^u_k$ to be fixed later. Then the function $h(\x)$ can be seen as a function with respect to $r$. By the Taylor formula, for $x \in \A(r)$ and $r$ close to $1$
\begin{align*}
h(\x) = 1 + \frac{\partial h}{\partial r}(\1) (r-1) + O((r-1)^2).
\end{align*}
Here $\x = \1$ means that for all $u \in [d], k \in [q], x^u_k = 1$. Since the first order dominates, the maximiser should attain the largest possible $\frac{\partial h}{\partial r}(\1)$. 
We calculate the derivative at first using the chain rule:
\begin{align*}
\frac{\partial h}{\partial r} &= \sum_{u=1}^d \sum_{j,k=2}^q d(U)^{d-1}\frac{\partial U}{\partial V_j} \frac{\partial V_j}{\partial x^u_k} \frac{\partial x^u_k}{\partial r} \\
&= \sum_{u=1}^d \sum_{j,k=2}^q d(U)^{d-1}\frac{\partial U}{\partial V_j} \frac{\partial V_j}{\partial x^u_k} \theta^u_k.
\end{align*}
The derivative at fixed point $\x = \1$ is 
\begin{align*}
\frac{\partial U}{\partial V_j}(\1)  = -d^{-1}(A/B)\delta_{2,j}, \qquad  \frac{\partial V_j}{\partial x^u_k}(\1) =  -d^{-1}(A/B)\delta_{j,k}, 
\end{align*}
where we use the following notation for the simplification 
\begin{align}\label{eq.AB}
A:= d(1-p), \qquad B:= p+q-1,
\end{align}
and the condition  $p  \in [1 - \frac{q}{d+1}, 1) \cap (0,1)$ implies that 
\begin{align}\label{eq.AB2}
    A \leq B.
\end{align}
Therefore, we have 
\begin{align*}
\frac{\partial h}{\partial r}(\1) &= d^{-2} (A/B)^2  \sum_{u=1}^d \sum_{j,k=2}^q \theta^u_k \delta_{2,j} \delta_{j,k}  \\
&=  d^{-2} (A/B)^2  \sum_{u=1}^d \sum_{k=2}^q \theta^u_k \delta_{2,k}  \\
& \leq d^{-1} (A/B)^2.
\end{align*}
The equality is attained when $\theta^u_2 = 1$ for all $u \in [d]$. This implies that for all $u \in [d]$, $x^u_2 = r$.

\medskip

For other variables, we need more information, so we calculate the second derivative. Similarly, we have
\begin{align*}
h(\x) = 1 + \frac{\partial h}{\partial r}(\1) (r-1) + \frac{1}{2} \frac{\partial^2 h}{\partial r^2}(\1) (r-1)^2 +  O((r-1)^3),    
\end{align*}
and the maximiser also needs to maximize the second derivative. We calculate it
\begin{align*}
\frac{\partial^2 h}{\partial r^2} = \sum_{u_1,u_2=1}^d \sum_{k_1,k_2 = 2}^q \frac{\partial^2 h}{\partial x^{u_1}_{k_1} \partial x^{u_2}_{k_2}} \theta^{u_1}_{k_1} \theta^{u_2}_{k_2}.    
\end{align*}
The derivative $\frac{\partial^2 h}{\partial x^{u_1}_{k_1} \partial x^{u_2}_{k_2}}$ is 
\begin{align*}
\frac{\partial^2 h}{\partial x^{u_1}_{k_1} \partial x^{u_2}_{k_2}} &= \mathbf{I} + \mathbf{II} + \mathbf{III}, \\
\mathbf{I} &= d(d-1) (U)^{d-2}\sum_{j_1, j_2 = 2}^q \Ll(\frac{\partial U}{\partial V_{j_1}}\Rr) \Ll(\frac{\partial U}{\partial V_{j_2}}\Rr)
\Ll(\frac{\partial V_{j_1}}{\partial x^{u_1}_{k_1}}\Rr)
\Ll(\frac{\partial V_{j_2}}{\partial x^{u_2}_{k_2}}\Rr),\\
\mathbf{II} &= d(U)^{d-1}\sum_{j_1, j_2 = 2}^q \Ll(\frac{\partial^2 U}{\partial V_{j_1}\partial V_{j_2}}\Rr) 
\Ll(\frac{\partial V_{j_1}}{\partial x^{u_1}_{k_1}}\Rr)
\Ll(\frac{\partial V_{j_2}}{\partial x^{u_2}_{k_2}}\Rr),\\
\mathbf{III} &= d(U)^{d-1}\sum_{j=2}^q \Ll(\frac{\partial U}{\partial V_{j}}\Rr)
\Ll(\frac{\partial^2 V_{j}}{\partial x^{u_1}_{k_1} \partial x^{u_2}_{k_2}}\Rr).\\
\end{align*}
Notice that 
\begin{align*}
\frac{\partial^2 U}{\partial V_{j_1}\partial V_{j_2}}(\1) &= d^{-1}B^{-1}(A/B)(\delta_{2,j_1} + \delta_{2,j_2}), \\
\frac{\partial^2 V_{j}}{\partial x^{u_1}_{k_1} \partial x^{u_2}_{k_2}}(\1) &= \Ll\{\begin{array}{cc}
    d^{-1}B^{-1}(A/B)(\delta_{j,k_1} + \delta_{j,k_2}) &  u_1 = u_2,\\
    d^{-2}(A/B)^2 \delta_{j,k_1} \delta_{j,k_2} & u_1 \neq u_2.
\end{array}\Rr.
\end{align*}
We put these back to the three terms $\mathbf{I}, \mathbf{II}$ and $\mathbf{III}$ and obtain that
\begin{align*}
\mathbf{I}(\1) &= d^{-2}(1-d^{-1})(A/B)^4 \delta_{2,k_1}\delta_{2,k_2},\\
\mathbf{II}(\1) &= d^{-2}B^{-1}(A/B)^3(\delta_{2,k_1} + \delta_{2,k_2}), \\
\mathbf{III}(\1) &= \Ll\{\begin{array}{cc}
     -2d^{-1}B^{-1}(A/B)^2 &  u_1 = u_2,\\
     -d^{-2}(A/B)^3 \delta_{2,k_1}\delta_{2,k_2}&  u_1 \neq u_2.
\end{array}\Rr.     
\end{align*}
We do a sum and this gives us 
\begin{equation}\label{eq.hD2}
\begin{split}
&\frac{\partial^2 h}{\partial r^2}(\1) \\
&=  \Ll(\sum_{u_1,u_2=1}^d   \theta^{u_1}_{2} \theta^{u_2}_{2}\bigg(d^{-2}(1-d^{-1})(A/B)^4  - d^{-2}(A/B)^3 (1 - \delta_{u_1, u_2})\bigg)\Rr) \\
& \quad + \Ll(2d^{-2}B^{-1}(A/B)^3\sum_{u_1,u_2 = 1}^d  \theta^{u_1}_2\Ll(\sum_{k=2}^q \theta^{u_2}_k\Rr) - 2d^{-1}B^{-1}(A/B)^2 \sum_{u_2=1}^d\Ll( \sum_{k=2}^q \theta^{u_2}_k\Rr)^2 \Rr).        
\end{split}    
\end{equation}
Notice that the maximiser satisfies that $\theta^u_2 = 1$ for all $u \in [d]$, the second line of \eqref{eq.hD2} is 
\begin{multline}\label{eq.hD2.part1}
\sum_{u_1,u_2=1}^d   \theta^{u_1}_{2} \theta^{u_2}_{2}\bigg(d^{-2}(1-d^{-1})(A/B)^4  - d^{-2}(A/B)^3 (1 - \delta_{u_1, u_2})\bigg)\\
=(1-d^{-1})(A/B)^4 - (1-d^{-1})(A/B)^3,
\end{multline}
which is constant. For the third line of \eqref{eq.hD2}, we also use this property to replace $\theta^{u_1}_2=1$ in it
\begin{align*}
&2d^{-2}B^{-1}(A/B)^3\sum_{u_1,u_2 = 1}^d  \theta^{u_1}_2\Ll(\sum_{k=2}^q \theta^{u_2}_k\Rr) - 2d^{-1}B^{-1}(A/B)^2 \sum_{u_2=1}^d\Ll( \sum_{k=2}^q \theta^{u_2}_k\Rr)^2\\
&=2d^{-2}B^{-1}(A/B)^3\sum_{u_1,u_2 = 1}^d  \Ll(\sum_{k=2}^q \theta^{u_2}_k\Rr) - 2d^{-1}B^{-1}(A/B)^2 \sum_{u_2=1}^d\Ll( \sum_{k=2}^q \theta^{u_2}_k\Rr)^2\\
&=2d^{-1}B^{-1}(A/B)^2\sum_{u = 1}^d  \Ll((A/B)  \Ll(\sum_{k=2}^q \theta^{u}_k\Rr) - \Ll(\sum_{k=2}^q \theta^{u}_k\Rr)^2\Rr).
\end{align*}
Notice that $A \leq B$ and $\sum_{k=2}^q \theta^{u}_k \geq \theta^u_2 = 1$, the the quadratic function 
\begin{align*}
\Ll(\sum_{k=2}^q \theta^{u}_k\Rr) \mapsto (A/B)  \Ll(\sum_{k=2}^q \theta^{u}_k\Rr) - \Ll(\sum_{k=2}^q \theta^{u}_k\Rr)^2,    
\end{align*}
is decreasing with respect to $\Ll(\sum_{k=2}^q \theta^{u}_k\Rr)$. Therefore, the maximum is attained when $\theta^u_k = 0$ for $k \in [d] \setminus \{2\}$ and all $u \in [d]$, i.e. $x^u_k = 1$ for $k \in [d] \setminus \{2\}$ and all $u \in [d]$. This is the desired result.
\end{proof}

\subsection{Proof of Theorem~\ref{thm.main}}
With the discussions above, we can now prove our main theorem.
\begin{proof}[Proof of Proposition~\ref{prop.TwoStep}]
Proposition~\ref{prop.TwoStep} follows directly from Proposition~\ref{prop.TwoStepBound} and Proposition~\ref{prop.expansion}. 
\end{proof}
\begin{proof}[Proof of Theorem~\ref{thm.main}] Using \eqref{eq.ffLimit} and Proposition~\ref{prop.TwoStep}, we have
\begin{align*}
\lim_{n \to \infty}  \Ll((r^*_{n+2}-1)^{-2} - (r^*_{n}-1)^{-2}\Rr) \geq \lim_{n \to \infty}  \Ll(((f \circ f)(r^*_{n})-1)^{-2} - (r^*_{n}-1)^{-2}\Rr) = \frac{d^2-1}{3d^2}.
\end{align*}
Then we obtain an estimate
\begin{align*}
\liminf_{n \to \infty} \frac{1}{2n+2} (r^*_{2n+2}-1)^{-2}
&= \liminf_{n \to \infty} \frac{1}{2n+2} \Ll(\sum_{k=1}^{n} \Ll((r^*_{2k+2}-1)^{-2} - (r^*_{2k}-1)^{-2}\Rr)\Rr) \\
&\geq \lim_{k \to \infty} \frac{1}{2}\Ll((r^*_{2k+2}-1)^{-2} - (r^*_{2k}-1)^{-2}\Rr)\\
&= \frac{d^2-1}{6d^2}.
\end{align*}
Similar estimate also works for the odd terms, which implies that 
\begin{align}\label{eq.ratioRate2}
\liminf_{n \to \infty} \frac{1}{n} (r^*_{n}-1)^{-2} \geq \frac{d^2-1}{6d^2}.   
\end{align}
This estimate can also be integrated to that of probability. An easy comparison gives us 
\begin{multline*}
(q-1)(r^*_n)^{-1} \max_{\xi} \mu^\xi_{n,p_c,q}[\sigma_* = 1] + \max_{\xi} \mu^\xi_{n,p_c,q}[\sigma_* = 1]    \leq 1  \\
\leq (q-1)(r^*_n) \min_{\xi} \mu^\xi_{n,p_c,q}[\sigma_* = 1] + \min_{\xi} \mu^\xi_{n,p_c,q}[\sigma_* = 1],
\end{multline*}
which implies that 
\begin{align*}
\frac{1}{1 + (q - 1)r^*_n}  \leq  \min_{\xi}\mu^\xi_{n,p_c,q}[\sigma_* = 1] \leq \frac{1}{q} \leq  \max_{\xi}\mu^\xi_{n,p_c,q}[\sigma_* = 1] \leq \frac{r^*_n}{r^*_n + q - 1}.
\end{align*}
Therefore, for the maximum deviation from $\frac{1}{q}$, we have 
\begin{align*}
\max_\xi \Ll\vert \mu^\xi_{n,p_c,q}[\sigma_* = 1] - \frac{1}{q}\Rr\vert &= \max\Ll\{\max_{\xi}\mu^\xi_{n,p_c,q}[\sigma_* = 1] - \frac{1}{q}, \frac{1}{q} - \min_{\xi}\mu^\xi_{n,p_c,q}[\sigma_* = 1]\Rr\}\\
&\leq \max\Ll\{\frac{(q-1)(r^*_n-1)}{q(r^*_n + q - 1)}, \frac{(q-1)(r^*_n-1)}{q(1 + (q - 1)r^*_n)}\Rr\}.
\end{align*}
This together with \eqref{eq.ratioRate2} gives us 
\begin{align*}
&\liminf_{n \to \infty} \frac{1}{n}\max_\xi \Ll\vert \mu^\xi_{n,p_c,q}[\sigma_* = 1] - \frac{1}{q}\Rr\vert^{-2} \\
&\geq \liminf_{n \to \infty} \frac{1}{n}\max\Ll\{\frac{(q-1)(r^*_n-1)}{q(r^*_n + q - 1)}, \frac{(q-1)(r^*_n-1)}{q(1 + (q - 1)r^*_n)}\Rr\}^{-2}\\
&= \frac{d^2-1}{6d^2}  \Ll(\frac{q^2}{q-1}\Rr)^2.
\end{align*}
This is in fact the upper bound for the convergence rate of the marginal probability. On the other hand, we have proved the lower bound of the convergence rate in Proposition~\ref{prop.pure}
\begin{multline*}
    \limsup_{n \to \infty} \frac{1}{n}\max_\xi \Ll\vert \mu^\xi_{n,p_c,q}[\sigma_* = 1] - \frac{1}{q}\Rr\vert^{-2} \\
    \leq \lim_{n \to \infty} \frac{1}{n} \Ll(  \Ll \vert \mu^{1}_{n,p_c,q}[\sigma_* = 1] - \frac{1}{q} \Rr \vert \Rr)^{-2} = \frac{d^2-1}{6d^2}\Ll(\frac{q^2}{q-1}\Rr)^2. 
\end{multline*}
Thus we finish the proof of the main theorem.
\end{proof}
\begin{remark} Follow the same proof of Theorem~\ref{thm.main}, we can also establish the convergence rate for the subcritical case: under Hypothesis~\ref{hyp} for $p  \in (1 - \frac{q}{d+1}, 1) \cap (0,1)$, we have an exponential convergence
 \begin{align*}
    \lim_{n \to \infty} \frac{1}{n}\log \Ll(\max_{\xi}  \Ll\vert \mu^\xi_{n,p,q}[\sigma_* = 1] - \frac{1}{q}\Rr\vert\Rr) = \log \Ll(\frac{d(1-p)}{p+q-1}\Rr).
 \end{align*}
\end{remark}

\section{Convergence under the double periodic condition}
In this section we prove Proposition \ref{prop.TwoStepfm}, which implies that the two-step iteration $(f \circ f)$, where $f$ is defined in \eqref{eq.f}, drives any initial condition to the fixed point $r=1$.  Proposition \ref{prop.TwoStepfm} was applied to prove the convergence rate for pure boundary conditions in Proposition \eqref{prop.pure}. 

We give an elementary proof for Proposition \ref{prop.TwoStepfm} which can be read independently from the rest of the paper. In fact, we will study the following more general iteration $(f_m \circ f_m)$, where
\begin{equation}\label{eq.fm}
f_m(x) := \Ll(\frac{(p+q-1)+ (m-1+p)(x-1)}{(p+q-1)+m(x-1)}\Rr)^d.    
\end{equation}
Obviously, from the definition \eqref{eq.f}, $f$ is the case of $f_m$ for $m = 1$. The definition of $f_m$ is not artificial, because $(f_m \circ f_m)(r)$ is the mapping $\hat{h}(r)$ in \eqref{eq.hNew} under the similar case of \eqref{eq.Condition3} and $\A(r)$ in \eqref{eq.Ar}: for every $v,u \in [d]$, $\{x^{vu}_k\}_{k \in [q]}$ is the same vector, and there are $m$ entries taking the value $r$ while $(q-m)$ entries taking the value $1$. It implies the two-step iterations under a family of boundary conditions converges to the unique fixed point, and we wish it will be helpful to prove the uniqueness conjecture (Hypothesis \ref{hyp}) of the AF Potts model on trees in the future.

The main result of this part is the following proposition.

\begin{proposition}\label{prop.TwoStepfm}
For any $x \geq 1$ and any integer $1 \leq m \leq q-1$, we have an estimate that $(f_m \circ f_m)'(x) \in \Ll(0, \Ll(\frac{d(1-p)}{p+q-1}\Rr)^2\Rr]$, and $(f_m \circ f_m)'(x) = \Ll(\frac{d(1-p)}{p+q-1}\Rr)^2$ admits a unique solution in $[1,\infty)$ that $x=1$.
\end{proposition}

From Proposition~\ref{prop.TwoStepfm} and since we assume the subcritical or critical phase, we have $\frac{d(1-p)}{p+q-1} \in (0,1]$. Then the iteration gives a strictly decreasing sequence once starting from some initial value above $1$, and eventually converges to $1$ as it is the unique fixed point; see Corollary~\ref{cor.rn} for details. 

\begin{remark}
Proposition~\ref{prop.TwoStepfm} covers the special case of the uniqueness conjecture for a class of the ``double periodic boundary conditions", namely, for the AF Potts model on a regular tree of height $n$, we assume the boundary conditions are given by $d^2$ identical copies of the boundary conditions for sub-trees of height $n-2$. One strategy to prove the uniqueness conjecture in future is to approximate any Gibbs measures on trees by double periodic Gibbs measures. 
\end{remark}

We first give the strategy of the proof of Proposition~\ref{prop.TwoStepfm}. Recall the constant $A = d(1-p)$ and $B=p+q-1$ defined in \eqref{eq.AB}, then we have ${f_m(x) := \Ll(\frac{B+ (m-1+p)(x-1)}{B+m(x-1)}\Rr)^d}$ and
\begin{align*}
(f_m \circ f_m)'(x) &= \frac{A^2}{B^2}\Ll(\frac{B^2}{(B + m(x-1))(B + m(f_m(x)-1))}\Rr)^2 \\
& \qquad \times \Ll(\frac{(B + (m-1+p)(x-1))(B + (m-1+p)(f_m(x)-1))}{(B + m(x-1))(B + m(f_m(x)-1))}\Rr)^{d-1} .    
\end{align*}

We introduce two functions 
\begin{equation}\label{eq.HKm}
\begin{split}
H_m(x) &:= \bigg(B + m(x-1)\bigg)\bigg(B + m(f_m(x)-1)\bigg) - B^2,\\
K_m(x) &:= \bigg(B + m(x-1)\bigg)\bigg(B + m(f_m(x)-1)\bigg) \\
& \qquad - \bigg(B + (m-1+p)(x-1)\bigg)\bigg(B + (m-1+p)(f_m(x)-1)\bigg). 
\end{split}
\end{equation}
Keep in mind that $A = B$ for the critical case, so 
 $(f_m \circ f_m)'(x) \in \Ll(0, \Ll(\frac{d(1-p)}{p+q-1}\Rr)^2\Rr]$ is equivalent to prove that 
\begin{align}\label{eq.Gm}
G_m(x) := \Ll(1 - \frac{H_m(x)}{B^2 + H_m(x)}\Rr)^2 \Ll(1 - \frac{K_m(x)}{B^2 + H_m(x)}\Rr)^{d-1} \leq 1,    
\end{align}
and the solution of $(f_m \circ f_m)'(x) = \Ll(\frac{d(1-p)}{p+q-1}\Rr)^2$ is that of $G_m(x)=1$. Thus, we focus on the function $G_m(x)$ and resume some elementary calculations in the following lemma. We also introduce 
\begin{align}\label{eq.gm}
g_m(x) := \frac{B+ (m-1+p)(x-1)}{B+m(x-1)},     
\end{align}
so it will help simplify the notation.

\begin{lemma}\label{lem.pre}
\begin{enumerate}
\item Using the notation $g_m$ in \eqref{eq.gm}, we have 
\begin{equation}\label{eq.fgm}
\begin{split}
f_m = (g_m)^d, \qquad f_m' = d(g_m)^{d-1}g_m',\\
f_m'' = d(d-1)(g_m)^{d-2}(g_m')^2 + d (g_m)^{d-1}g_m'',     
\end{split}
\end{equation}
where the expression of $g_m'$ and $g_m''$ are 
\begin{align}\label{eq.gmDff}
g_m'(x) = \frac{B(p-1)}{(B+ m(x-1))^2}, \qquad g_m''(x) = \frac{2mB(1-p)}{(B+ m(x-1))^3}.    
\end{align}
We also have $f_m(1) = 1$ and $f'_m(1) = -A/B$.
\item Some properties about $H_m(x)$: for all $1 \leq m \leq q-1$, we have
\begin{itemize}
    \item $H_m(1) = 0$ and $H_m(x) > 0$ for all $x > 1$. 
    \item $H_m'(1) = mB(1 - A/B)$, $H_m'(x) > 0$ for all $x > 1$.
    \item The explicit expression for $H_m''(x)$ is 
    \begin{align}
        H_m''(x) = m (g_m)^{d-2}\bigg(B+ m(x-1)\bigg)^{-4}\bigg(\gamma_H + \beta_H(x-1)\bigg),
    \end{align}
    with the quantity $\beta_H, \gamma_H$ defined as 
    \begin{align}
        \beta_H &:= (1-d^{-1})m A^2 B^2, \label{eq.betaH} \\
        \gamma_H &:= (1-d^{-1})A^2 B^3.\label{eq.gammaH}
    \end{align}
\end{itemize}
\item Some properties about $K_m(x)$: denote by 
\begin{align}\label{eq.Cm}
    C_m : = 2m-1+p, 
\end{align}
then we have
\begin{itemize}
    \item $K_m(1) = 0$. 
    \item $K_m'(1) = (1-p)B(1 - A/B)$.
    \item The explicit expression for $K_m''(x)$ is 
    \begin{align}
        K_m''(x) = (1-p) (g_m)^{d-2}\bigg(B+ m(x-1)\bigg)^{-4}\bigg(\gamma_K + \beta_K(x-1)\bigg),
    \end{align}
    with the quantity $\beta_K, \gamma_K$ defined as 
    \begin{align}
        \beta_K &:= AB^2(-2(C_m - m)^2 + (1-d^{-1})AC_m), \label{eq.betaK} \\
        \gamma_K &:= AB^3(A-C_m). \label{eq.gammaK}
    \end{align}
\end{itemize}
\end{enumerate}
\end{lemma}
\begin{proof}
\begin{enumerate}
    \item This is the direct calculation of $f_m$ in function of $g_m$. 
    \item $H_m(1) = 0$ comes direct from the fact that $f_m(1) = 1$. We calculate the derivative of $H_m$, which is 
    \begin{align*}
    H_m'(x) &=  m\Ll(\bigg(B + m(f_m(x)-1)\bigg) + f_m'(x)\bigg(B + m(x-1)\bigg)\Rr)\\
    &= m\bigg((B-m)(f_m'(x) + 1) + m(f_m(x) + x f_m'(x))\bigg).
    \end{align*}
    We evaluate it at $1$ using $f_m(1) = 1$ and $f_m'(1) = -A/B$, then we obtain that ${H_m'(1) = mB(1-A/B)}$. We calculate the second derivative 
    \begin{align*}
    H_m''(x) &= m\bigg((B-m)f_m''(x) + m(2f_m'(x) + x f_m''(x))\bigg)\\
    &= m\bigg((B+m(x-1))f_m''(x) + 2m f_m'(x) \bigg).
    \end{align*}
    We put in the expression of $f_m''(x)$ in function of $g_m(x)$ and its derivative and obtain that 
    \begin{align*}
    H_m''(x) &= m\bigg((B+m(x-1))f_m''(x) + 2m f_m'(x) \bigg)\\
    &= m\bigg((B+m(x-1))\Ll( d(d-1)(g_m)^{d-2}(g_m')^2 + d (g_m)^{d-1}g_m''\Rr) + 2m (d(g_m)^{d-1}g_m') \bigg)\\
    &= m(g_m)^{d-2}\bigg((B+m(x-1))\Ll( d(d-1)(g_m')^2 + d g_m'' g_m \Rr) + 2m (d g_m' g_m) \bigg).
    \end{align*}
    We keep the first factor $(g_m)^{d-2}$, while put in the exact expression of $g_m, g_m'$ and $g_m''$ in the other terms
    \begin{align*}
    H_m''(x) &=  m(g_m)^{d-2}(B + m(x-1))^{-4} \times \\
    & \Bigg((B+m(x-1))\bigg( (1-d^{-1})A^2 B^2 + (2m)AB\big(B+(m-1+p)(x-1)\big) \bigg)\\
    & \qquad \Ll. - (2m)AB \bigg(\big(B+(m-1+p)(x-1)\big)\big(B+m(x-1)\big)\Bigg) \Rr)\\
    &= m(g_m)^{d-2}(B + m(x-1))^{-4}\bigg((1-d^{-1})A^2 B^2 (B+m(x-1))\bigg).
    \end{align*}
    This gives the expression of $\beta_H$ and $\gamma_H$, and also proves that
    \begin{align*}
    \forall  x \geq 1, \qquad   H_m''(x) \geq m B^{-4}\bigg((1-d^{-1})A^2 B^3\bigg) > 0.    
    \end{align*}
    Combined this with the expression of $H_m'(1)$, we know that $H_m'(x), H_m(x) > 0$ for all $x > 1$. 
    \item The part about the $K_m$ is similar. $K_m(1) = 0$ is also the result from $f_m(1) = 1$. We calculate its derivative that 
    \begin{align*}
    K_m'(x) &=  m \bigg(B + m(f_m(x)-1)\bigg)  + m f_m'(x)\bigg(B + m(x-1)\bigg) \\
    \qquad &- (m-1+p)\bigg(B + (m-1+p)(f_m(x)-1)\bigg)  - (m-1+p) f_m'(x)\bigg(B + (m-1+p)(x-1)\bigg) \\
    &= (1-p)f_m'(x)\bigg(B + (2m-1+p)(x-1)\bigg) + (1-p)\bigg(B + (2m-1+p)(f_m(x)-1)\bigg).
    \end{align*}
    Recall that $C_m = 2m-1+p$, we have 
    \begin{equation}\label{eq.KmD1}
    \begin{split}
    K_m'(x) &=  (1-p)f_m'(x)\bigg(B + C_m(x-1)\bigg) + (1-p)\bigg(B + C_m(f_m(x)-1)\bigg)\\  
    &=(1-p)\bigg((B-C_m)(f_m'(x) + 1) + C_m(f_m(x) + x f_m'(x))\bigg).    
    \end{split}
    \end{equation}
    $K_m'(1) = (1-p)B(1-A/B)$ is a direct result from ${f_m'(1) = -A/B}$ and ${f_m(1) = 1}$. Then we calculate the second derivative that 
    \begin{align*}
    K_m''(x) &= (1-p)\bigg((B-C_m)f_m''(x) + C_m(2f_m'(x) + x f_m''(x))\bigg)\\
    &= (1-p)\bigg((B+C_m(x-1))f_m''(x) + 2C_m f_m'(x) \bigg).
    \end{align*}
    We put in the expression of $f_m''(x)$ in function of $g_m(x)$ and its derivative and obtain that 
    \begin{align*}
    K_m''(x) &= (1-p)\bigg((B+C_m(x-1))\Ll( d(d-1)(g_m)^{d-2}(g_m')^2 + d (g_m)^{d-1}g_m''\Rr) + 2C_m (d(g_m)^{d-1}g_m') \bigg)\\
    &= (1-p)(g_m)^{d-2}\bigg((B+C_m(x-1))\Ll( d(d-1)(g_m')^2 + d g_m'' g_m \Rr) + 2C_m (d g_m' g_m) \bigg).
    \end{align*}
    Once again, we keep the first factor $(g_m)^{d-2}$ and put in the exact expression of $g_m, g_m'$ and $g_m''$ in the other terms
    \begin{align*}
    K_m''(x) &=  (1-p)(g_m)^{d-2}(B + m(x-1))^{-4} \times \\
    & \Bigg((B+C_m(x-1))\bigg( (1-d^{-1})A^2 B^2 + (2m)AB\big(B+(m-1+p)(x-1)\big) \bigg)\\
    & \qquad \Ll. - (2C_m)AB \bigg(\big(B+(m-1+p)(x-1)\big)\big(B+m(x-1)\big)\Bigg) \Rr).
    \end{align*}
    From this, the term $(x-1)^2$ vanishes and we get the desired $\beta_K, \gamma_K$.
\end{enumerate}
\end{proof}
\begin{remark}
We remark here that the monotonicity of $K_m$ is not as good as that of $H_m$, because its sign of second derivative depends on the regime of $x$ and the parameters; see the expression of $\gamma_K, \beta_K$. 
\end{remark}

With the preparation of Lemma~\ref{lem.pre}, we are now ready to prove Proposition~\ref{prop.TwoStepfm}.
\begin{proof}[Proof of Proposition~\ref{prop.TwoStepfm}]
We aim to prove \eqref{eq.Gm}. Since we always have $H_m(x) \geq 0$ for $x \geq 1$ from Lemma~\ref{lem.pre}, the key is the sign of $K_m(x)$. For the intervals $K_m(x) \geq 0$, it is clear $G_m(x) \leq 1$ from the expression \eqref{eq.Gm}, so we are interested in the part $K_m(x) \leq 0$. We calculate the derivative of $G_m(x)$ in function of $K_m(x)$ and $H_m(x)$
\begin{equation}\label{eq.GmDiff}
\begin{split}
G_m' &= \Ll(1 - \frac{H_m}{B^2 + H_m}\Rr) \Ll(1 - \frac{K_m}{B^2 + H_m}\Rr)^{d-2} \\
& \qquad \times\Bigg( (d-1)\Ll(\frac{-K_m'(B^2 + H_m)+K_m H_m'}{(B^2 + H_m)^2}\Rr)\Ll(\frac{B^2}{B^2 + H_m}\Rr) \\
& \qquad \qquad + 2 \Ll(\frac{B^2+H_m-K_m}{B^2 + H_m}\Rr)\Ll( \frac{-B^2 H_m'}{(B^2 + H_m)^2}\Rr)\Bigg)\\
&= \Ll(1 - \frac{H_m}{B^2 + H_m}\Rr) \Ll(1 - \frac{K_m}{B^2 + H_m}\Rr)^{d-2}\Ll(\frac{B^2}{(B^2 + H_m)^3}\Rr) \times\\
& \qquad \Ll(-(B^2+H_m)\big((d-1)K_m'+2H_m'\big) + (d+1)H_m'K_m\Rr).    
\end{split}
\end{equation}
The sign of $G_m'$ depends on the last line. Here one important observation is 
\begin{align}\label{eq.HKObservation}
\forall x > 1, \qquad (d-1)K_m'(x)+2H_m'(x) > 0.    
\end{align}
We verify this observation by writing the sum as integration of second order derivative from Lemma~\ref{lem.pre}
\begin{equation}\label{eq.ObserDiff}
\begin{split}
&(d-1)K_m'(x)+2H_m'(x) \\
&= (d-1)K_m'(1)+2H_m'(1) \\
& \quad + \int_{1}^x  (g_m)^{d-2}\bigg(B+ m(y-1)\bigg)^{-4}\bigg((d-1)(1-p)\big(\gamma_K + \beta_K(y-1)\big) + 2 m \big(\gamma_H + \beta_H (y-1) \big)\bigg) \, \d y.
\end{split}
\end{equation}
From Lemma~\ref{lem.pre}, we know that $K_m'(1), H_m'(1) \geq 0$ as $A \leq B$ from \eqref{eq.AB2}. Then, we apply the explicit expression \eqref{eq.gammaH}, \eqref{eq.gammaK}, \eqref{eq.betaH}, \eqref{eq.betaK}
\begin{align*}
(d-1)(1-p) \gamma_K + 2m \gamma_H &= (d-1)(1-p)AB^3(A-C_m) + 2m(1-d^{-1})A^2B^3\\
&=(1-d^{-1})A^2B^3(A-C_m + 2m) \\
&> 1-p > 0.
\end{align*}
Here we use the expression $A = d(1-p)$ and ${C_m = 2m-1+p \leq 2m}$ in the last line. Similarly, for the coefficient of $(y-1)$ we have 
\begin{align*}
&(d-1)(1-p) \beta_K + 2m \beta_H \\
&= (d-1)(1-p)AB^2(-2(C_m - m)^2 + (1-d^{-1})AC_m)  + 2m^2(1-d^{-1})A^2B^2\\
&= (1-d^{-1})A^2B^2\bigg(2m^2-2(C_m - m)^2 + (1-d^{-1})AC_m\bigg)\\
&> 0.
\end{align*}
In the last line, we use the fact $C_m - m = 2m-1+p -m \in (0,m)$. Therefore, we justify the observation \eqref{eq.HKObservation}. 

\medskip

We now use this observation to conclude the proof. Because $K_m$ is a continuous function, the part $x \geq 1$ can be decomposed into the intervals of $K_m(x) \geq 0$ and $K_m(x) < 0$, i.e.
\begin{equation}\label{eq.Inteval}
\begin{split} 
&x_1^+ := 1,\\
\forall i \in \N^+, \quad & x_i^- := \inf\{x: x \geq x_i^+, K_m(x) < 0\}, \\
\forall i \in \N^+, i\geq 2, \quad  & x_i^+ := \inf\{x: x \geq x_{i-1}^-, K_m(x) \geq 0\},\\
\forall i \in \N^+, \quad & I_i^+:=[x_i^+, x_i^-], \qquad I_i^-:= (x_i^-, x_{i+1}^+).
\end{split}
\end{equation}
Then we have 
\begin{align*}
\{x: x \geq 1, K_m(x) \geq 0\} = \bigcup_{i \in \N^+} I_i^+, \qquad \{x: x \geq 1, K_m(x) < 0\} = \bigcup_{i \in \N^+} I_i^-.    
\end{align*}
As stated in the previous paragraphs that,  $H_m(x) \geq 0$ for $x \geq 1$ from Lemma~\ref{lem.pre} and the expression \eqref{eq.Gm} imply that $G_m(x) \leq 1$ for $K_m(x) \geq 0$. The observation \eqref{eq.HKObservation} and $H_m'(x) \geq 0$ for $x \geq 1$ from Lemma~\ref{lem.pre} give that that, for the case $K_m(x) < 0$, the last line in  \eqref{eq.GmDiff} is strictly negative, and $G_m'(x) < 0$. Therefore, we obtain that 
\begin{align}\label{eq.GmDecreasing}
\forall i \in \N^+, \forall x \in I_i^- = (x_i^-, x_{i+1}^+), \qquad G_m(x) < G_m(x_i^-),
\end{align}
and 
\begin{align*}
    \sup_{x \geq 1} G_m(x) = \sup_{x \geq 1, K_m(x) \geq 0} G_m(x) \leq 1,
\end{align*}
which concludes \eqref{eq.Gm} and $(f_m \circ f_m)'(x) \in \Ll(0, \Ll(\frac{d(1-p)}{p+q-1}\Rr)^2\Rr]$.

Finally, we prove the unique solution of $(f_m \circ f_m)'(x) = \Ll(\frac{d(1-p)}{p+q-1}\Rr)^2$. As discussed above, it is equivalent to $G_m(x) = 1$ and is only possible to attain on $K_m(x) \geq 0$ thanks to the observation \eqref{eq.GmDecreasing}. Viewing the expression \eqref{eq.Gm}, this requires that $K_m(x) = 0$ and $H_m(x) = 0$. For the latter, the unique solution is $x=1$ from (2) of Lemma~\ref{lem.pre}, and $x=1$ also verifies the former equation. This finishes the proof of the uniqueness.
\end{proof}

In the end, we prove the convergence of $r_n$ at critical and subcritical regime as a simple application.
\begin{corollary}\label{cor.rn}
Let $q, d \in \N^+$ and $q \geq 2$. Also let $p  \in [1 - \frac{q}{d+1}, 1) \cap (0,1)$, $(r_n)_{n \in N^+}$ defined in \eqref{eq.ratio1} converges to $1$.
\end{corollary}
\begin{proof}
From Lemma~\ref{lem.pure}, we have $r_{2n+2} = (f \circ f)(r_{2n})$ with $f = f_1$ in \eqref{eq.fm} and for even terms $r_{2n} \geq 1$ by a simple induction. Then $p  \in [1 - \frac{q}{d+1}, 1) \cap (0,1)$ implies that $(f \circ f)'(x) \leq 1$ from Proposition~\ref{prop.TwoStepfm}. Thus, by the mean value theorem, 
\begin{align*}
\forall x \geq 1, \exists z \in (1,x), \text{ such that }  \qquad  (f \circ f)(x) - 1 = (f \circ f)'(z)(x-1) \leq x - 1,
\end{align*}
and the mapping $(f \circ f)$ admits a fixed point $x^*$

We prove by contradiction that $x^* = 1$ is the unique fixed point. Suppose that there is another fixed point $x^* > 1$ for $(f \circ f)$, then once again by the mean value theorem,
\begin{align*}    
\exists z \in (1,x), \qquad x^* - 1 = (f \circ f)(x^*) - 1 = (f \circ f)'(z)(x^*-1) \leq x^* - 1,
\end{align*}
which implies that $(f \circ f)'(z) = 1$. This contradicts Proposition~\ref{prop.TwoStepfm} that $1$ is the unique solution of $(f_m \circ f_m)'(x) = \Ll(\frac{d(1-p)}{p+q-1}\Rr)^2$.

Finally, the limit of $(r_{2n})_{n \in \N^+}$ follows the iteration $r_{2n+2} = (f \circ f)(r_{2n})$ and the unique fixed point of $(f \circ f)$ under the assumption $p  \in [1 - \frac{q}{d+1}, 1) \cap (0,1)$. The limit of $(r_{2n+1})_{n \in \N^+}$ follows that of the even terms and the continuity of $f$.  
\end{proof}

\subsection*{Acknowledgements}
The research of C.Gu and W.Wu are supported in part by the National Key R\&D Program of China (No. 2021YFA1002700). The research of W.Wu is supported in part by a Shanghai Municipal Education Commission grant. K. Yang would like to thank the NYU Shanghai and ECNU Mathematical Institute for hospitality. 

\bibliographystyle{abbrv}
\bibliography{Ref}

\end{document}